\documentclass{amsart}
\usepackage{amsmath}
\usepackage{amsfonts}
\usepackage{amssymb}
\usepackage{amsthm}
\usepackage[alphabetic]{amsrefs}
\usepackage{bbold}
\usepackage{colonequals}
\usepackage{enumerate}
\usepackage[T1]{fontenc}
\usepackage{bbm}
\usepackage{mathtools}
\usepackage{chngcntr}
\usepackage{url,xspace,adjustbox}
\usepackage{footmisc}
\usepackage{amssymb}
\usepackage{geometry}

\newcommand{\Frob}{\mathfrak{f}}
\newcommand{\GL}{\operatorname{GL}}
\newcommand{\Hom}{\operatorname{Hom}}
\newcommand{\codim}{\operatorname{codim}}
\newcommand{\Aut}{\operatorname{Aut}}
\newcommand{\Evs}{\operatorname{Evs}}
\newcommand{\Ev}{\operatorname{Ev}}

\newcommand{\RPhi}{\operatorname{R\hskip-1pt\Phi}}
\newcommand{\Der}{\operatorname{D}}
\newcommand{\Perv}{\operatorname{Per}}
\newcommand{\Loc}{\operatorname{Loc}}
\newcommand{\Rep}{\operatorname{Rep}}
\newcommand{\rank}{\operatorname{rank}}
\newcommand{\nocontentsline}[3]{}
\newcommand{\tocless}[2]{\bgroup\let\addcontentsline=\nocontentsline#1{#2}\egroup}
\newcommand{\Lgroup}[1]{\prescript{L}{}#1}
\newcommand{\CC}{\mathbb{C}}

\newcommand{\Mat}{\operatorname{Mat}}
\newcommand{\ABV}{{\mbox{\raisebox{1pt}{\scalebox{0.5}{$\mathrm{ABV}$}}}}}
\newcommand{\op}[1]{\operatorname{#1}}
\newcommand{\1}{{\mathbbm{1}}}
\newcommand{\dualgroup}[1]{\widehat{#1}}

\usepackage{adjustbox}
    
\newtheorem{theorem}{Theorem}[section]
\newtheorem{theorem*}{Theorem}

\theoremstyle{definition}
\newtheorem{definition}[theorem]{Definition}
\newtheorem{lemma}[theorem]{Lemma}

\newtheorem{example}[theorem]{Example}
\newtheorem{Remark}[theorem]{Remark}

\bibstyle{alphanumeric}

\usepackage{soul}

\usepackage{mathrsfs,stmaryrd}
\usepackage{yfonts, bbm}
\usepackage{enumitem}
\usepackage{float}

\usepackage{hyperref}
\hypersetup{
  colorlinks   = true, 
  urlcolor     = blue, 
  linkcolor    = red, 
  citecolor   = cyan 
}

\usepackage{tikz}
\usetikzlibrary{shapes,arrows,calc,matrix}
\usepackage{tikz-cd}
\usepackage{mathdots}

\usepackage{todonotes}
\usepackage{color}

\usepackage{amsmath}
\usepackage{lipsum}
\usepackage{setspace}

\counterwithin{table}{section}

\tikzset{
  symbol/.style={
    draw=none,
    every to/.append style={
      edge node={node [sloped, allow upside down, auto=false]{$#1$}}}
  }
}

\title[ ]{\small Proof of Vogan's conjecture on Arthur packets: irreducible parameters of $p$-adic general linear groups}
\author{Clifton Cunningham and Mishty Ray}
\date{}

\begin{document}

\begin{abstract}
In this paper we prove Vogan's conjecture on local Arthur packets, as recalled in \cite{CFMMX}*{Section 8.3, Conjecture 1(a)}, for irreducible Arthur parameters of $p$-adic general linear groups.
This result shows that these Arthur packets may be characterized by properties of simple perverse sheaves on a moduli space of Langlands parameters.  
\end{abstract}

\maketitle

\section{Introduction}

The local Langlands correspondence for a connected reductive algebraic group $G$ over a $p$-adic field $F$ partitions the set of equivalence classes of smooth irreducible representations of $G(F)$ into $L$-packets using equivalence classes of Langlands parameters.
Following \cites{Vogan:Langlands, ABV}, the correspondence may be viewed as a bijection between smooth irreducible representations that share a common infinitesimal parameter of a $p$-adic group $G(F)$, along with its pure inner forms, and simple objects in the category of equivariant perverse sheaves on a moduli space of Langlands parameters. 
This perspective leads to the notion of an ABV-packet, as articulated in \cite{CFMMX}*{Section 8}, which conjecturally generalizes the local Arthur packet - we call this as Vogan's conjecture on local Arthur packets. In this paper, we prove this conjecture for irreducible Arthur parameters of the $p$-adic group $GL_n(F)$. 

\subsection{Simple and irreducible parameters}
\label{simple parameters}
Let $F$ be a non-archimedean local field with residue characteristic $p$ and residue field $\mathbb{F}_q$, so $q=p^f$. Let  $W_F$ be the Weil group of $F$ and $W_F'$ denote the group $W_F \times \operatorname{SL}_2(\CC)$. 
Recall that a Langlands parameter is a homomorphism
\[\phi : W'_F \to \Lgroup{G} = \dualgroup{G}\rtimes W_F\]
commuting with the natural projections to $W_F$, satisfying conditions adapted from \cite{Borel:Corvallis}.
In particular, this means that all the information of a Langlands parameter $\phi$ is captured by a $1$-cocycle \[
\phi^0 : W'_F \to \dualgroup{G}
\]
such that $\phi(w,g) = \phi^0(g,w)\rtimes w$.
In the case $G= \GL_n$, the cocycle $\phi^0$ is a representation of the group $W'_F$.

Now set $W_F'':=W_F \times \operatorname{SL}_2(\CC) \times \operatorname{SL}_2(\CC)$ and recall that an Arthur parameter \cite{Arthur:book}*{Chapter 1} is a homomorphism
\[\psi : W''_F \to \Lgroup{G} = \dualgroup{G}\rtimes W_F,\]
likewise defined by a $1$-cocycle 
\[
\psi^0 : W'_F \to \dualgroup{G}
\]
satisfying certain conditions, notably that the restriction of $\psi^0$ to $W_F$ is bounded. 
Again, when $G= \GL_n$, $\psi^0$ is a representation.
We say that $\psi$ is a {\it simple} Arthur parameter of $\GL_n(F)$ if $\psi^0$ is an irreducible representation and the restriction of $\psi^0$ to $W_F$ is trivial. We say that $\psi$ is \textit{irreducible} if $\psi^0$ is an irreducible representation of $W_F''$.

Every Arthur parameter determines a Langlands parameter by the rule
\[\phi_{\psi}(w,x)= \psi\left( w,x,\begin{pmatrix}
|w|^{\frac{1}{2}} & 0\\
0 & |w|^{-\frac{1}{2}}
\end{pmatrix}\right).\]

The infinitesimal parameter $\lambda_{\phi}$ of a Langlands parameter $\phi$ is a homomorphism \[\lambda_{\phi}: W_F \xrightarrow{}\Lgroup{G}\] defined by the rule
\[\lambda_{\phi}(w)= \phi\left( w,\begin{pmatrix}
|w|^{\frac{1}{2}} & 0\\
0 & |w|^{-\frac{1}{2}}
\end{pmatrix}\right).\]
See \cite{CFMMX}*{Section 4.1} for definition. As before, it is determined by
\[\lambda_{\phi}^0 : W_F \to \dualgroup{G}\] which is a representation when $G=\GL_n$. The function $\phi \mapsto \lambda_{\phi}$ defined on $\dualgroup{G}$-equivalence classes of these parameters is not injective in general. The infinitesimal parameters $\lambda_{\phi_{\psi}}$ for a simple Arthur parameter $\psi$ is unramified as a representation of $W_F$, \textit{i.e.}, it are trivial on the inertia group $I_F$.

\subsection{Local Langlands correspondence}
\label{local langlands correspondence}
The local Langlands conjecture for a connected reductive algebraic group $G$ over $F$ gives a map between the set $\Pi(G(F))$ of equivalence classes of smooth irreducible representations of $G(F)$ and the set $\Phi(\Lgroup{G})$ of equivalence classes of Langlands parameters. This map can be refined \cites{Vogan:Langlands, ABV} as a bijection between $\Pi(G(F))$ and the set $\Xi(\Lgroup{G})$ of equivalence classes of pairs $\xi = (\phi,\epsilon)$, called enhanced Langlands parameters, where $\epsilon$ is an irreducible representation of the finite group
\[
S_\phi := Z_{\dualgroup{G}}(\phi)/Z_{\dualgroup{G}}(\phi)^0 Z(\dualgroup{G})^{\Gamma_F},
\]
where $\Gamma_F$ is the absolute Galois group of $F$.
We emphasize the distinction between Langlands parameters and $L$-parameters - Langlands parameters are maps and $L$-parameters are $\widehat{G}$-equivalence classes of Langlands parameters. These $L$-parameters partition $\Pi(G(F))$ into disjoint finite sets $\Pi_{\phi}(G(F))$ called $L$-packets. This correspondence must satisfy a list of functoriality properties and should be compatible with $L$-functions and $\epsilon$-factors on both sides (See \cite{Borel:Corvallis} or \cite{Kal}, for example). 
When $G=\GL_n$, the local Langlands correspondence is a theorem \cites{Hen, Harris-Taylor}. In this case, $S_{\phi}=1$ (see \cite{Arthur:book}, for example) and the correspondence can be viewed as a  bijection 
\[\Pi(G(F)) \xrightarrow{} \Phi(G),\] where $\Phi(G)$ is the set of $L$-parameters for $G$. Thus, $L$-packets are singletons for $\GL_n$.

\subsection{A-packets}
\label{apackets}
In 1989, Arthur introduced \cite{arthur1989unipotent} what are now known as Arthur packets, or $A$-packets, which enjoy some nice properties that non-tempered $L$-packets fail to satisfy. This was a global theory and the local analogue was clarified by Arthur's subsequent work \cites{Arthur:book} in 2013. 
Local $A$-packets $\Pi_\psi(G(F))$ are finite sets of smooth irreducible representations of $G(F)$ which are not necessarily disjoint and consist of unitary admissible representations. The representations that appear in $A$-packets are said to be of Arthur type. 
Every A-packet $\Pi_\psi(G(F))$ contains a distinguished $L$-packet $\Pi_{\phi_\psi}(G(F))$ through the injection injection $\psi \mapsto \phi_\psi$. 
In the 1990s, Adams, Barbasch, and Vogan suggested \cite{ABV} a geometric characterization of $A$-packets by using microlocal analysis on certain stratified complex varieties built out of Langlands parameters. When $G=\GL_n$, $A$-packets are singletons \cite{Arthur:book}; in particular, the $A$-packet for $\psi$ coincides with the $L$-packet for $\phi_{\psi}$.

\subsection{Vogan-Langlands correspondence}
\label{vogan langlands corresp}
Vogan's perspective \cite{Vogan:Langlands} on the local Langlands correspondence suggests a beautiful geometric approach to $A$-packets. This perspective concerns a moduli space $X_{\lambda}$ of Langlands parameters that have the same infinitesimal parameter $\lambda$. Vogan observed that simple objects in the category $\Perv_{\widehat{G}}(X_{\lambda})$ of equivariant perverse sheaves on $X_{\lambda}$ are naturally identified with the pairs $(\phi, \epsilon)$, when we properly interpret the finite group attached to $\phi$ and consider representations of $G(F)$ together with its pure inner forms. Viewed in this way, the local Langlands correspondence takes the form of a bijection 
\[
\Pi_{\lambda}(G/F) \to \Perv_{H_{\lambda}}(V_{\lambda})_{\op{/iso}}^{\op{simple}}
\]
between the set $\Pi_{\lambda}(G/F)$ of equivalence classes of irreducible objects in the category $\Rep_{\lambda}(G/F)$ of smooth representations of $G(F)$ and its pure rational forms which have a matching infinitesimal parameter $\lambda$ and the set $\Perv_{H_{\lambda}}(V_{\lambda})_{\op{/iso}}^{\op{simple}}$ of isomorphism classes of simple objects in the category $\Perv_{\widehat{G}}(X_{\lambda})$.





Instead of working with $\Perv_{\widehat{G}}(X_{\lambda})$, we work with an equivalent category $\Perv_{H_{\lambda}}(V_{\lambda})$ (see \cite{CFMMX}*{Section 4.5}). The variety $V_{\lambda}$ of Langlands parameters and the group $H_{\lambda}$  are described in \ref{modspace}. 
In the case $G=\GL_n$, the correspondence can be viewed as a bijection \[\Pi_{\lambda}(G(F)) \to \Perv_{H_{\lambda}}(V_{\lambda})_{\op{/iso}}^{\op{simple}},\] 
where $\Pi_{\lambda}(G(F))$ is the set of equivalence classes of smooth irreducible representations of $G(F)$ with a matching infinitesimal parameter $\lambda$ and
\[\Perv_{H_{\lambda}}(V_{\lambda})_{\op{/iso}}^{\op{simple}}=\{\mathcal{IC}(\1_C): C \subseteq V_{\lambda} \text{ is an } H_{\lambda}\text{-orbit} \},
\]
where $\1_C$ denotes the constant local system on $C$.
%
In this way, every $\pi \in \Pi_{\lambda}(G(F))$ corresponds to $\mathcal{P}(\pi)=\mathcal{IC}(\1_C) \in \Perv_{H_{\lambda}}(V_{\lambda})_{\op{/iso}}^{\op{simple}}$, for a unique $H_\lambda$-orbit $C\subseteq V_\lambda$. In other words, in this case there is a bijection between elements of $\Pi_{\lambda}(G(F))$ and $H_{\lambda}$-orbits $C$ in $V_{\lambda}$. 
                     
\subsection{Vogan's conjecture on A-packets}
      
Following an approach similar to \cite{ABV}, Vogan attached to any Langlands parameter $\phi$ the set of irreducible representations of $G(F)$ and its pure inner forms for which the characteristic cycles of the corresponding simple perverse sheaf contains the conormal bundle of the $\widehat{G}$-orbit of $\phi$ in the moduli space mentioned above. This notion is revised in \cite{CFMMX} in terms of vanishing cycles. This set is referred to as the ABV-packet $\Pi_\phi^\ABV(G/F)$.

Vogan's perspective on the local Langlands correspondence introduces pure A-packets $\Pi_\psi(G/F)$ which are unions of $A$-packets for $G$ and its pure inner forms. This perspective also identifies $\Pi_\psi(G/F)$  with a finite set of simple perverse sheaves on $V_{\lambda_\psi}$, where $\lambda_\psi$ is the infinitesimal parameter of $\phi_\psi$.
What property do these perverse sheaves share, and how are they determined by the Arthur parameter $\psi$?
This question is answered by the functor
\begin{equation*}
    \label{evs}
    \Evs_{\psi}: \Perv_{H_{\lambda_\psi}}(V_{\lambda_\psi}) \xrightarrow{} \Rep(A_\psi)
\end{equation*}
where $A_\psi = Z_{\dualgroup{G}}(\psi)/Z_{\dualgroup{G}}(\psi)^0$, 
constructed in \cite{CFMMX}*{Section 7.10}. 
In fact, this functor is defined in terms of a more general functor $\Ev_{C_\psi}: \Der_{H_{\lambda_\psi}}(V_{\lambda_\psi}) \to \Der_{H_{\lambda_\psi}}(\Lambda^\text{reg}_{C_\psi})$ 
where $\Lambda_{C_\psi} \subset \Lambda_{\lambda_\psi}$ is the conormal bundle above $C_\psi$ in the conormal variety for $V_\lambda$ \cite{CFMMX}*{Section 7.3}. 
This $\Ev_{C_\psi}$ functor is defined in such a way that
\[
\left(\Ev_{C_\psi}\mathcal{P}\right)_{(x,\xi)}
=
\left(\RPhi_\xi \mathcal{P}\right)_x,
\]
for all $(x,\xi)\in \Lambda^\text{reg}_{C_\psi}$, where $\RPhi_\xi$ is Deligne's vanishing cycles functor for $\xi : V_{\lambda_\psi} \to \CC$ \cite{CFMMX}*{Equation 7.6}.
When restricted from $\Der_{H_{\lambda_\psi}}(V_{\lambda_\psi})$ to $\Perv_{H_{\lambda_\psi}}(V_{\lambda_\psi})$, the functor $\Ev_{C_\psi}$ lands in $\Perv_{H_{\lambda_\psi}}(\Lambda^\text{reg}_{C_\psi})$, which is equivalent to the semisimple category $\Loc_{H_{\lambda_\psi}}(\Lambda^\text{reg}_{C_\psi})$ of equivariant local systems on $\Lambda^\text{reg}_{C_\psi}$, by \cite{CFMMX}*{Proposition 6.9}. 
This proposition also gives an equivalence
\[
\Loc_{H_{\lambda_\psi}}(\Lambda^\text{reg}_{C_\psi}) \cong \Rep(A_\psi);
\]
see \cite{CFMMX}*{Equation 7.30}.
As explained in \cite{CFMMX}*{Equations 7.22, 7.24}, the precise relation between $\Evs_\psi$ and $\Ev_{C_\psi}$ is given by
\begin{equation}\label{equation:EvsEv}
\Evs_\psi \mathcal{P} = \mathcal{H}^{1+\codim C_\psi }_{(x_\psi,\xi_\psi)}\left(\Ev_{C_\psi} \mathcal{P}\right),
\end{equation}
as a vector space with $A_\psi$-action, where $(x_\psi,\xi_\psi)\in \Lambda_{\lambda_\psi}$ is defined in \cite{CFMMX}*{Equation 6.12, 6.13}.
For $G= \GL_n$, $A_\psi =1$ by \cite{Arthur:book}*{Chapter 1}, for example; so in this case, $\Rep(A_\psi)$ is simply the category of vector spaces. 
Following \cite{CFMMX}, the ABV-packet for $\psi$ is defined by
\begin{equation}
    \label{abvpacket}
    \Pi_{\phi_\psi}^{\ABV}(G(F)):=\{\pi \in \Pi_{\lambda_\psi}(G(F)): \operatorname{Evs}_{\psi}(\mathcal{P}(\pi))\neq 0 \},
\end{equation} 
where $\mathcal{P}(\pi)$ is the simple perverse sheaf on $V_{\lambda_\psi}$ corresponding to $\pi\in \Pi_{\lambda_\psi}(G(F))$.

The main result of this paper, below, gives a proof of \cite{CFMMX}*{Section 8.3, Conjecture 1(a)} adapted to general linear groups, for irreducible Arthur parameters:

\begin{theorem*}[Theorem \ref{bow wow main theorem}]
\label{maintheorem}
Let $\psi$ be an irreducible Arthur parameter of $G=\GL_n(F)$, let $\phi_{\psi}$ be its corresponding Langlands parameter, let $\Pi_{\psi}(G(F))$ be the Arthur packet attached to $\psi$. 
Then 
\[ \Pi_{\psi}(G(F))= \Pi_{\phi_{\psi}}^{\ABV}(G(F)).\]
\end{theorem*}

We briefly overview the structure of the paper. Henceforth, $G=\GL_n$. In Section \ref{modspace}, we write down the \textit{Vogan variety} $V_{\lambda}$, the group $H_{\lambda}$, and the explicit action of $H_{\lambda}$ on $V_{\lambda}$. Our perspective on the local Langlands correspondence is that of representations corresponding to orbits in the Vogan variety. Both sides of this correspondence are encoded in combinatorial objects - multisegments for representations and rank triangles for orbits. We reiterate these notions and write down an algorithm to go from multisegments to rank triangles and vice-verse in Section \ref{multi-ranktriangle}. In Section \ref{zel-inv}, we discuss the Zelevinsky involution on multisegments and explain how it corresponds to an involution on orbits. We write down the Moeglin-Waldspurger algorithm to compute the involution on a multisegment, and use this notion to prove Lemma \ref{mainlemma}. In section \ref{proof}, we adapt Lemma \ref{mainlemma} to a statement about orbits in Lemma \ref{orbitduallemma}. Using this lemma and the properties of the functor $\op{Evs_\psi}$, we prove the conjecture for simple $A$-parameters in Theorem \ref{ wow main theorem}. In section \ref{irreducible a parameter section}, we use the process of hyper-unramification to generalize the result to an irreducible $A$-parameter. 

In upcoming work, we expect to generalize Theorem \ref{maintheorem} to arbitrary $A$-parameters for $\GL_n(F)$ for a non-archimedean field $F$.

Since ABV-packets are expected to generalize Arthur packets, it is natural to ask - are all ABV-packets are singletons for $\GL_n$? The answer is no; one can see an ABV-packet of size $2$ for a non Arthur-type Langlands parameter of $\GL_{16}$ in \cite{KS_sing}.

\subsection*{Acknowledgements} 
The authors would like to acknowledge Andrew Fiori for sharing his insights on the combinatorics of multisegments, which was crucial for Lemma \ref{mainlemma}. The second named author would like to extend her gratitude to Kristaps Balodis and James Steele for helpful discussions. 


\section{A moduli space of Langlands parameters}
\label{modspace}
Recall that $F$ is a non-archimedean local field of residue characteristic $q=p^f$ with $p$-adic absolute value denoted by $|\cdot|$ and the uniformizer denoted by $\varpi$ . The Weil group is $W_F=I_F \rtimes \mathbb{Z}$, where $I_F$ is the inertia subgroup and $\mathbb{Z}$ is identified with the cyclic subgroup generated by the frobenius. From local class field theory, we have the local Artin map $\alpha: W_F \xrightarrow{} F^{\times}$ that maps $I_F$ to the group of units $\mathfrak{o}^{\times}$ in $F^{\times}$ and the geometric frobenius $\Frob$ to $1/\varpi$. We abuse notation and use $|\cdot|$ to denote the absolute value on $W_F$, where $|w| := |\alpha(w)|$. Thus, $|\Frob|=|1/\varpi|=q$. We have that $G=\GL_n(F)$, and we can work with $\widehat{G}=\GL_n(\mathbb{C})$ instead of $\Lgroup{G}$. 

We consider $\textit{simple}$ Arthur parameters, {\it i.e.}, of the form 
\begin{align}
\label{arthurpar}
\psi: W_F'' &\longrightarrow \operatorname{GL}_n(\mathbb{C}) \notag\\
(w,x,y) &\mapsto \operatorname{Sym}^d(x)\otimes \operatorname{Sym}^a(y),
\end{align}
where $a$ and $d$ are positive integers \footnotemark. 
\footnotetext{Here we use $d$ and $a$ in our notation to refer to the fact that the first $\operatorname{SL}_2$ in $W_F''=W_F\times \operatorname{SL}_2(\mathbb{C}) \times \operatorname{SL}_2(\mathbb{C})$ is the ``Deligne-$\operatorname{SL}_2$'', {\it i.e.}, it comes from Weil-Deligne representations, and the second $\operatorname{SL}_2$ is the ``Arthur-$\operatorname{SL}_2$'', {\it i.e.}, it comes from the definition of Arthur parameters.} 
For $w \in W_F$, let $d_w$ denote the matrix 
$\begin{pmatrix}
|w|^{\frac{1}{2}} & 0\\
0 & |w|^{-\frac{1}{2}}
\end{pmatrix}$. The corresponding Langlands parameter is given by
\begin{align}
\label{langlandspar}
\phi_{\psi}: W_F' &\longrightarrow \operatorname{GL}_n(\mathbb{C}) \notag\\
(w,x) &\mapsto \psi(w,x,d_w).
\end{align}
One can simplify $\psi(w,x,d_w)$ to get 
\[\phi_{\psi}(w,x)= |w|^{\frac{a}{2}}\operatorname{Sym}^{d}(x) \oplus |w|^{\frac{a-2}{2}}\operatorname{Sym}^{d}(x) \cdots \oplus |w|^{\frac{-a}{2}}\operatorname{Sym}^{d}(x). \] 
The infinitesimal parameter associated to this Langlands parameter is given by
\begin{align}
\label{infpar}
\lambda_{\phi_{\psi}}: W_F &\longrightarrow \operatorname{GL}_n(\mathbb{C})\notag\\
(w,x) &\mapsto \phi_{\psi}(w,d_w).
\end{align}
One can simplify $\phi_{\psi}(w,d_w)$ to get 
\begin{equation}\label{lambda eigenvalues}
\lambda_{\phi_{\psi}}(w)= \bigoplus_{i=0}^{a}\bigoplus_{j=0}^{d} |w|^{\frac{a+d}{2}-(i+j)} 
\end{equation}
We fix this infinitesimal parameter for the rest of this work and call it $\lambda$ for simplicity. The element $\lambda(\Frob) \in \GL_{n}(\mathbb{C})$ is semisimple with eigenvalues $\lambda_i := q^{e_i}$ occurring with multiplicities $m_i$, where $0\leq i \leq k$, with $e_i=e_{i-1}-1$. Let us rewrite the matrix $\lambda(\Frob)$ in a more instructive way:
%

\begin{equation}
    \label{eq:inf}
    \lambda(\Frob) = \begin{bmatrix} 
    q^{e_0}I_{m_0} & 0 & \dots \\
    0 & q^{e_1}I_{m_1} & \dots\\
    \vdots & \ddots & \\
    0 &        & q^{e_k}I_{m_k} 
    \end{bmatrix}.
\end{equation} 
We denote the corresponding eigenspaces by $E_i$. If we consider the $n$-dimensional vector space which is the standard representation of $\operatorname{Lie}(\widehat{G})$ spanned by the usual basis $\{f_1,...,f_n\}$, then $E_0$ is spanned by $\{f_1,...,f_{m_0}\}$, $E_1$ by $\{f_{m_0+1},...,f_{m_1}\}$, and in general $E_i$ by $\{f_{m_{i-1}+1},...,f_{m_i}\}$. 

Following \cite{Vogan:Langlands} and \cite{CFMMX}, we describe the variety $V_{\lambda}$ and group $H_{\lambda}$ below: 
\begin{equation}
\label{vlambdadef}
    V_{\lambda} :=\{X \in \mathfrak{gl}_n(\CC): \operatorname{Ad}(\lambda(w))X=|w|X, \text{ }\forall w \in W_F\}.
\end{equation}
As the absolute value on $I_F$ is trivial, any $X \in V_{\lambda}$ is determined by the equation
\begin{equation}
\label{vlambdaequation}\operatorname{Ad}(\lambda(\Frob))X=qX.\end{equation}
$V_{\lambda}$ is referred to as the \textit{Vogan variety}. In this case, it is easy to see the structure of this variety. Any matrix in $V_{\lambda}$ is strictly upper triangular or trace zero so it can be identified as a vector subspace of the $\mathfrak{sl}_n$. $\lambda(\Frob)$ is an element inside the maximal torus in $\operatorname{SL}_n$. Using the Cartan decomposition of $\mathfrak{sl_n}$, $V_{\lambda}$ can be identified as a sum of root spaces of $\mathfrak{sl_n}$ such that the roots evaluate to $q$ at $\lambda(\Frob)$. This shows that $V_{\lambda}$ is simply an affine space. This fact is true more generally, as seen in \cite{CFMMX}*{Lemma 5.5}.

\begin{lemma}
The Vogan variety $V_{\lambda}$ can be decomposed as
\begin{equation}
\label{vlambda}   
V_{\lambda} = \Hom(E_{k},E_{k-1}) \times \Hom(E_{k-1},E_{k-2})\times \cdots \times \Hom(E_1,E_0). 
\end{equation}

\end{lemma}
\begin{proof}
Let $X \in V_{\lambda}$ and $v \in E_i$. By definition, $\lambda(\Frob)v=q^{e_i}v$ and $\lambda(\Frob)X\lambda(\Frob)^{-1}=qX$. We show that $Xv \in E_{i-1}$. Indeed
\begin{align*}
    \lambda(\Frob) Xv &= \lambda(\Frob)X\lambda(\Frob)^{-1}\cdot \lambda(\Frob)v,\\
                     &= qX\cdot q^{e_i}v,\\
                     &= q^{e_i +1}Xv.
\end{align*}
As $E_{i-1}$ is the eigenspace for $q^{e_i + 1}$,
\[V_{\lambda} \subseteq \Hom(E_{k},E_{k-1}) \times \Hom(E_{k-1},E_{k-2})\times \cdots \times \Hom(E_1,E_0).\]
Now we show the reverse containment. Let $T \in \Hom(E_i,E_{i-1})$. If we view this as an endomorphism of $\mathbb{C}^n$, then the $m_{i-1} \times m_i$ matrix associated with $T$ can be thought of as an $n \times n$ matrix $(a_{ij})$ such that $a_{ij} \neq 0$ if $m_{i-2}+1\leq i\leq m_{i-1} \text{ and }m_{i-1}+1\leq j\leq m_i$, and $a_{ij} = 0$ otherwise. However, the adjoint (conjugation) action of $\lambda(\Frob)$ simply multiplies all the nonzero entries $a_{ij}$ by $q$. Thus, $\operatorname{Ad}(\lambda(\Frob))\cdot (a_{ij}) = q \cdot (a_{ij})$ and the matrix $(a_{ij})$ can be identified with an element of $V_{\lambda}$. This shows the reverse containment.  
\end{proof}

The group $H_{\lambda}$ is $Z_{\widehat{G}}(\lambda)$, the centralizer of the image of $\lambda$ in $\widehat{G}$. Again, any $g \in H_{\lambda}$ is determined by the equation
\begin{equation}
\label{hlambdaequation}
     \operatorname{Ad}(\lambda(\Frob))\cdot g=g.
\end{equation}
Using the description of $\lambda(\mathfrak{f})$ in $\eqref{eq:inf}$, we can directly see that 


\begin{equation*}
    Z_{\widehat{G}}(\lambda(\Frob)) = \begin{bmatrix} 
    \GL_{m_0}(\mathbb{C}) & 0 & \dots \\
    0 & \GL_{m_1}(\mathbb{C}) & \dots\\
    \vdots & \ddots & \\
    0 &        & \GL_{m_k}(\mathbb{C}) 
    \end{bmatrix}
\end{equation*} 
We may identify $H_{\lambda}$ as follows.
\begin{equation}
    \label{hlambda}
    H_{\lambda} = \Aut(E_k) \times \Aut(E_{k-1}) \times \cdots \times \Aut(E_0)
\end{equation}
$H_{\lambda}$ acts on $V_{\lambda}$ via conjugation. Using the interpretation from $\eqref{vlambda}$ and $\eqref{hlambda}$, we describe the action as follows.
For elements in $V_{\lambda}$ let us use the notation
\[
x = (x_{kk-1},\cdots,x_{21}, x_{10}), \text{ where } x_{ii-1} \in \Mat_{m_{i-1}, m_i}(\mathbb{C}).
\]
For elements of $H_{\lambda}$, we write
\[
h = (h_k,\cdots h_1,h_0),  \text{ where } h_i\in \GL_{m_i}(\mathbb{C}).
\]
In this notation, the action $H_{\lambda}$ on  $V_{\lambda}$ is given by
\begin{equation}
    \label{action}
    (h_k, \ldots , h_1,h_0)  \cdot  (x_{kk-1}, \ldots, x_{21}, x_{10})=(h_{k-1} x_{kk-1} h_k^{-1}, \ldots ,h_1 x_{21} h_2^{-1}, h_0 x_{10} h_1^{-1} ).
\end{equation}
Observe that 
$
\text{rank}(h_{i-1}x_{ii-1}h_i^{-1})=\text{rank}(x_{ii-1}),
$ thus any orbit preserves the ranks of $x_{ii-1}$s. More generally every orbit is determined by the ranks of $x_{ii-1}$s and all their permitted products. Realizing elements of $V_{\lambda}$ as a matrix as in \eqref{vlambda}, we have 
\begin{equation*}
    x = \begin{bmatrix} 
    0 & x_{10} & 0      & \dots \\
    0 & 0      & x_{21} & \dots\\
    \vdots &  &\ddots & \\
     &     &   & x_{kk-1}\\
    0 & 0   & \dots & 0
    \end{bmatrix}.
\end{equation*} 
The orbit of any $x$ is determined by its Jordan canonical form. Notice that $0$ is the only eigenvalue, so calculating dimensions of the generalized eigenspaces will involve computing ranks of powers of this matrix $x$, which in turn involves computing ranks of all permitted products \begin{equation}
\label{product}
   x_{ij}:=x_{ii-1}x_{i-1i-2}\cdots x_{j+1j} 
\end{equation} 
where where $0\leq j < i-1\leq k$. 
Writing down these rank equations exactly realizes an orbit as an algebraic variety. 

With emphasis on the fact that $\lambda$ determines $H_{\lambda}$ and $V_{\lambda}$, we drop the subscript and use $H$ and $V$ instead.
\begin{example}
\label{h and v}
For $\GL_4$, set $\psi(w,x,y)=\op{Sym}^1(x)\otimes\op{ Sym}^1(y)$. Then $\phi_{\psi}=\psi(w,x,d_w)=|w|^{1/2}x\oplus |w|^{-1/2}x$ and $\lambda(w)=\psi(w,d_w,d_w)$, so
\begin{equation*}
    \lambda(\Frob) = \begin{bmatrix} 
    q & 0 & 0 & 0 \\
    0 & 1 & 0 & 0\\
    0 & 0 & 1 &  0\\
    0 & 0 & 0 &  q^{-1}\\
    \end{bmatrix}.
\end{equation*} 
By direct calculation using \eqref{vlambdaequation} and \eqref{hlambdaequation}
\begin{align*}
    V &=\left\{\begin{bmatrix} 
    0 & u_1 & u_2 & 0 \\
    0 & 0 & 0 & v_1\\
    0 & 0 & 0 & v_2\\
    0 & 0 & 0 & 0\\
    \end{bmatrix} : u_1,u_2,v_1,v_2 \in \CC \right\} \simeq \mathbb{A}^4_{\CC},\\
    H &=\left\{\begin{bmatrix} 
    t_1 & 0 & 0 & 0 \\
    0 & t_2 & r & 0\\
    0 & s & t_3 &  0\\
    0 & 0 & 0 & t_4\\
    \end{bmatrix} : t_1,t_4 \in \CC^{\times},\begin{bmatrix} t_2 & r\\
    s & t_3\end{bmatrix} \in \GL_2(\CC) \right\} \simeq \GL_1(\CC) \times \GL_2(\CC) \times \GL_1(\CC). \\
\end{align*}
The eigenvalues of $\lambda(\Frob)$ are $q, 1, q^{-1}$  with multiplicities $1,2,$ and $1$ and eigenspaces \\$E_0=\op{span}\{(1,0,0,0)\}$, $E_1=\op{span}\{(0,1,0,0),(0,0,1,0)\}$ and $E_2=\op{span}\{(0,0,0,1)\}$, respectively. In the formulation of \eqref{vlambda} and \eqref{hlambda}, we may write
\begin{align*}
    V &= \op{Hom}(E_2,E_1) \times \op{Hom}(E_1,E_0),\\
    H &= \op{Aut}(E_2) \times \op{Aut}(E_1) \times \op{Aut}(E_0).
\end{align*}
Following \eqref{action}, the action of $H$ on $V$ is given by
\[\left(t_4,\begin{bmatrix}t_2&r\\s&t_3\end{bmatrix}, t_1\right)\cdot \left(\begin{bmatrix}v_1\\v_2\end{bmatrix}, [u_1,u_2]\right)=\left(\begin{bmatrix}t_2&r\\s&t_3\end{bmatrix}\begin{bmatrix}v_1\\v_2\end{bmatrix}t^{-1}_4, t_1[u_1,u_2]\begin{bmatrix}t_2&r\\s&t_3\end{bmatrix}^{-1} \right).\]
\end{example}


\section{From multisegements to rank triangles and vice-versa}
\label{multi-ranktriangle}
As outlined in Section \ref{vogan langlands corresp}, there is a bijection between equivalence classes of smooth irreducible representations that share the infinitesimal parameter $\lambda$ and $H$-orbits in $V$. We know that smooth irreducible representations of $G$ can be parameterized using multisegments \cite{Z2}. The relationship between multisegments and orbits appears in \cite{zelevinskii1981p}. In this section, we introduce the combinatorial gadget of rank triangles to parametrize the orbits and write down an algorithm compute the multisegment corresponding to a rank triangle and vice-versa. This has been worked out for some specific cases in $\GL_{16}$ in \cite{KS_sing}; here we present it with updated notation for $\GL_n$. 

\subsection{Multisegments}
The Langlands classification using multisegments is summarized beautifully in \cite{kudla1994local}*{Section 1}. We recall this classification in more detail following the notation in \cite{kudla1994local}. For any representation $\pi$ of $\GL_m(F)$, let $\pi(i):=|\text{det}(\cdot)|^i\pi$. For a partition $n=\underbrace{m+m+\ldots+m}_{r\text{-times}}$ and a supercuspidal representation $\sigma$ of $\GL_m(F)$, we call \begin{equation}
    \label{segment}
    (\sigma, \sigma(1), \ldots, \sigma(r-1))=[\sigma,\sigma(r-1)]=\Delta
\end{equation} a segment. This segment determines a representation on a parabolic subgroup of $G$ whose levi is $\underbrace{\GL_m(F) \times \GL_m(F) \times \cdots \times \GL_m(F)}_{r\text{-times}}$. More precisely, we have the representation $\sigma \otimes \sigma(1) \otimes \cdots \otimes \sigma(r-1)$ on the levi, inflated to a representation of the parabolic, which is trivial on the unipotent. We can then carry out parabolic induction to obtain the induced representation $I_P^{G}(\Delta)$ of $G$ which has a unique irreducible quotient denoted by $Q(\Delta)$. 

A multisegment is a collection of segments with repetitions allowed. The Langlands classification theorem tells us that any smooth irreducible representation of $G$ occurs as a unique irreducible quotient inside a parabolically induced representation determined by a multisegment $\{\Delta_1, \Delta_2, \ldots, \Delta_r\}$ subject to certain conditions - we denote that quotient by $Q(\Delta_1, \Delta_2, \ldots, \Delta_r)$. Next, for integers $i<j$ we introduce the notation 
\begin{equation}
    \label{oursegment}
    [i,j]:=(|\cdot|^i, |\cdot|^{i+1}, \ldots, |\cdot|^j)
\end{equation}
 for a segment which is the special case of \eqref{segment} when we consider the partition $1+1+ \cdots +1$ and $\sigma$ to be the character $|\cdot|^i$ of $F^{\times}$. This notation may be extended to half integers $i<j$ as long as $j-i+1$ is a positive integer (this is the length of the segment). A segment of length 1 of the form $\{|\cdot|^i\}$ is just denoted $[i]$. There is a partial order on multisegments which we recall below.
\begin{definition}
\label{multiorder}
Let $\alpha$ and $\beta$ be any two multisegments. We say that $\alpha \leq \beta$ if we can form $\beta$ by performing elementary operations on segments in $\alpha$. More precisely, $\alpha \leq \beta$ if we can form $\beta$ by replacing any two segments $\Delta_1$ and $\Delta_2$ in $\alpha$ with
\[ \begin{cases} 
      \Delta_1 \cup \Delta_2 \text{ and } \Delta_1 \cap \Delta_2 & \text{ if }\Delta_1 \cap \Delta_2 \neq \emptyset,\\
      \Delta_1 \cup \Delta_2 & \text{ if } \Delta_1 \cup \Delta_2 = \emptyset \text{ and } \Delta_1 \cup \Delta_2 \text{ is a segment},\\
      \Delta_1 \text{ and } \Delta_2& \text{ otherwise. }
   \end{cases}
\]
\begin{example}
 In the case of $\GL_4$, the Steinberg representation can be written as $Q([-\frac{3}{2},\frac{3}{2}])$ and the trivial representation is given by $Q([\frac{3}{2}], [\frac{1}{2}], [-\frac{1}{2}], [-\frac{3}{2}])$. The multisegment for $\phi_{\psi}$ from Example \ref{h and v} is \{[0,1],[-1,0]\}. 
 \end{example}
\end{definition}

\subsection{Rank triangles}
We know from Section \ref{modspace} that any orbit of \\ $(x_{kk-1},\ldots, x_{21}, x_{10}) \in V$ is determined by the ranks of $x_{ij}$s as defined in \eqref{product}. We arrange these ranks in a rank triangle. This arrangement helps us write down a visual algorithm to compute the correspondence between multisegments and rank triangles in Section \ref{ranktri-multiseg}. To that end we set
\[r_{ij}:=\op{rank}(x_{ij}),\]
where $0\leq j<i\leq k$. 
Recall that $E_i$ is an eigenspace for $\lambda(\Frob)$ with eigenvalue $q^{e_i}$ with multiplicity $m_i$. We arrange these exponents, multiplicities, and ranks into a triangle to reflect the corresponding combinations of the $x_{ij}$ and refer to this as a {\it rank triangle}:
\begin{center}
\begin{tikzpicture}
\node(A) at (0,0.5) {$e_k$};
\node(B) at (2,0.5) {$e_{k-1}$};
\node(C) at (4,0.5) {$\cdots$};
\node(L) at (6, 0.5) {$e_1$};
\node(M) at (8,0.5) {$e_0$};
\draw (-0.25,0.25) -- (8.25,0.25);
\node(A) at (0,0) {$m_k$};
\node(B) at (2,0) {$m_{k-1}$};
\node(C) at (4,0) {$\cdots$};
\node(L) at (6,0) {$m_1$};
\node(M) at (8,0) {$m_0$};
\draw (-0.25,-0.25) -- (8.25,-0.25);
\node(D) at (1,-1) {$r_{kk-1}$};
\node(E) at (3,-1) {$\hspace{0.3 cm} r_{(k-1)(k-2)}\cdots $};
\node(d) at (5,-1) {$r_{21}$};
\node(e) at (7,-1) {$r_{10}$};
\node(F) at (2,-2) {$r_{k(k-2)}$};
\node(G) at (4,-2) {$\cdots$};
\node(H) at (6,-2) {$r_{20}$};
\node(I) at (3,-3) {$\cdots$};
\node(J) at (5.5,-3) {$\cdots$};
\node(K) at (4.25,-4) {$r_{k0}$};.
\end{tikzpicture}
\end{center}
From left to right, the values in the top row of the rank triangle correspond, respectively, to the exponents $e_i$ of the eigenvalues $q^{e_i}$. The values in the second row correspond respectively to the multiplicities $m_i$ of the eigenvalues $q^{e_i}$. 
The row below these multiplicities shows ranks $r_{ij}$ of $x_{ij}$, which are subject to the condition that every rank is less than or equal to the two eigenvalue multiplicities above it. For the other rows, the ranks are subject to exactly two conditions: 

\begin{itemize}
    \item Every rank is less than or equal to the two ranks above it. This is a consequence of a basic fact from linear algebra: if $A$ and $B$ are matrices such that $AB$ and $BC$ are defined then $\rank(AB) \leq \rank(A)$ and $\rank(AB) \leq \rank(B)$. 
    \item The ranks satisfy 
     \[
r_{is} + r_{lj} \leq r_{ij} + r_{ls}, \qquad j \leq s < l \leq i.
\]This condition is a consequence of the Frobenius inequality from linear algebra: if $A,B$ and $C$ are matrices such that $AB$, $BC$ and $ABC$ are defined then $\rank(AB) + \rank(BC) \leq \rank(ABC) + \rank(B)$.
\end{itemize}

 The set of $H$-orbits in $V$ is naturally in bijection with rank triangles subject to these conditions. 

The set of $H$-orbits in $V$ carries a partial order defined by the Zariski topology: 

\begin{definition}
\label{orbitorder}
There is a partial order on the set of orbits in  $V_{\lambda}$: For any two orbits $C$ and $C'$, $C\leq C'$ if and only if $C \subseteq \overline{C'}$. 

\end{definition}
This partial order can be read easily from the corresponding rank triangles as follows:
If the ranks for $C$ are $r_{ij}$ and $C'$ are $r'_{ij}$ then $C\leq C'$ if and only if $r_{ij}\leq r'_{ij}$ for all $i,j$.

\begin{example} \label{rank triangle gl4}
 For $\lambda$ in Example \ref{h and v}, the template for the rank triangles will look as follows.
 \begin{center}
\begin{tikzpicture}
\node(A) at (0,0.5) {$e_2$};
\node(B) at (1,0.5) {$e_1$};
\node(C) at (2,0.5) {$e_0$};
\draw (-0.25,0.25) -- (2.25,0.25);
\node(A) at (0,0) {$m_2$};
\node(B) at (1,0) {$m_1$};
\node(C) at (2,0) {$m_0$};
\draw (-0.25,-0.25) -- (2.25,-0.25);
\node(D) at (0.5,-0.5) {$r_{21}$};
\node(E) at (1.5,-0.5) {$r_{10}$};
\node(d) at (1,-1) {$r_{20}$};
\end{tikzpicture}
\end{center}
 The orbit corresponding to the case $r_{21}=r_{10}=1$ and $r_{20}=0$ will have the following rank triangle.
 \begin{center}
\begin{tikzpicture}
\node(A) at (0,0.5) {$-1$};
\node(B) at (1,0.5) {$0$};
\node(C) at (2,0.5) {$1$};
\draw (-0.25,0.25) -- (2.25,0.25);
\node(A) at (0,0) {$1$};
\node(B) at (1,0) {$2$};
\node(C) at (2,0) {$1$};
\draw (-0.25,-0.25) -- (2.25,-0.25);
\node(D) at (0.5,-0.5) {$1$};
\node(E) at (1.5,-0.5) {$1$};
\node(d) at (1,-1) {$0$};
\end{tikzpicture}
\end{center}
 
\end{example}

\subsection{Correspondence between rank triangles and multisegments}
\label{ranktri-multiseg}
We write down an algorithm to compute a rank triangle from a multisegment and vice-versa.

To each segment $\Delta = [e_{j},e_i]$ we associate the rank triangle $T_{\Delta}$ with
 \[ r_{ls} = \begin{cases} 1, & j\leq s<l\leq i; \\ 0, & \text{otherwise.} \end{cases} \qquad m_s = \begin{cases} 1, & j\leq s\leq i; \\ 0, & \text{otherwise}. \end{cases} \]
A multisegment $\underline{m}$ then determines the triangle 
 \[ T_{\underline{m}} = \sum_{\Delta\in \underline{m}} T_{\Delta}. \]
 
We now explain how to pass from a rank triangle with eigenvalues  $q^{e_i}$ and multiplicities (top row) $m_i$ to a multisegment with support $\underline{m}_\lambda$ by the following inductive rule:. 
 \begin{enumerate}
 \item Set $\underline{m}:=\emptyset$.
  \item Let $r_{ij}$ be the lowest and most to the left non-zero entry in the rank triangle. Note that $j< i$.
  \item Add the segment $\Delta = [e_i,e_j]$ to $\underline{m}$. 
  \item Replace $r_{k l}$ by $r_{k l}-1$ for each value of $k,l$ with $j\leq l< k\leq i$.
  Replace $m_k$ by $m_k-1$ for each $j\leq k\leq i$. That is, subtract $T_\Delta$ from the rank triangle.
  \item Repeat steps (ii) through (iv) until all $r_{ij}$ are zero.
  \item If any $m_k\ne 0$, add the singleton $\{ e_k\}$ with multiplicity $m_k$ to $\underline{m}$. 
  \end{enumerate}
 Now $\underline{m}$ is the multisegment determined by the rank triangle. These procedures, $\underline{m}\leftrightarrow  T_{\underline{m}}$ establish a bijection between rank triangles and multisegments.
 
 In Definition \ref{multiorder} we define a parital order on multisegments and in Definition \ref{orbitorder} we have a partial order on orbits (and therefore rank triangles). From \cite{zelevinskii1981p} (Theorem 2.2, Contiguity theorem) these partial orders are compatible under the correspondence. More precisely, let $\underline{m}, T_{\underline{m}}$, and $C_{\underline{m}}$ be a multisegment, its corresponding rank triangle and orbit respectively. Then the above statement can be rephrased as: If $\alpha$ and $\beta$ are multisegments, then $\alpha \leq \beta$ if and only if $C_{\alpha} \leq C_{\beta}$ if and only if $T_{\alpha} \leq T_{\beta}$. 

\begin{example} \label{multisegment-rk tri example}
Let us compute the multisegment for the rank triangle from Example \ref{rank triangle gl4}. 
 \begin{center}
\begin{tikzpicture}
\node(A) at (0,0.5) {$-1$};
\node(B) at (1,0.5) {$0$};
\node(C) at (2,0.5) {$1$};
\draw (-0.25,0.25) -- (2.25,0.25);
\node(A) at (0,0) {$1$};
\node(B) at (1,0) {$2$};
\node(C) at (2,0) {$1$};
\draw (-0.25,-0.25) -- (2.25,-0.25);
\node(D) at (0.5,-0.5) {$1$};
\node(E) at (1.5,-0.5) {$1$};
\node(d) at (1,-1) {$0$};
{\color{red}
\draw (0.5, -0.8) -- (-0.5,0.25);
\draw (0.5,-0.8) -- (1.5, 0.25);
\draw ((-0.5,0.25)-- (-0.5, 0.6);
\draw ((1.5,0.25)-- (1.5, 0.6);
\draw [-{Implies}, double] (1,-1.5) -- (1,-1.8);
}
\node(f) at (1,-2.5){$[-1,0]$};

\draw [-latex] (3,-0.25) -- (4,-0.25);
 
\node(A) at (5,0.5) {$-1$};
\node(B) at (6,0.5) {$0$};
\node(C) at (7,0.5) {$1$};
\draw (4.75,0.25) -- (7.25,0.25);
\node(A) at (5,0) {$0$};
\node(B) at (6,0) {$1$};
\node(C) at (7,0) {$1$};
\draw (4.75,-0.25) -- (7.25,-0.25);
\node(D) at (5.5,-0.5) {$0$};
\node(E) at (6.5,-0.5) {$1$};
\node(d) at (6,-1) {$0$};
{\color{red}
\draw (6.5, -0.8) -- (5.5,0.25);
\draw (6.5,-0.8) -- (7.5, 0.25);
\draw ((5.5,0.25)-- (5.5, 0.6);
\draw ((7.5,0.25)-- (7.5, 0.6);
\draw [-{Implies}, double] (6,-1.5) -- (6,-1.8);
}
\node(f) at (6,-2.5){$[0,1]$};

\draw [-latex] (8,-0.25) -- (9,-0.25);

\node(A) at (10,0.5) {$-1$};
\node(B) at (11,0.5) {$0$};
\node(C) at (12,0.5) {$1$};
\draw (9.75,0.25) -- (12.25,0.25);
\node(A) at (10,0) {$0$};
\node(B) at (11,0) {$0$};
\node(C) at (12,0) {$0$};
\draw (9.75,-0.25) -- (12.25,-0.25);
\node(D) at (10.5,-0.5) {$0$};
\node(E) at (11.5,-0.5) {$0$};
\node(d) at (11,-1) {$0$};
{\color{red}
\draw [-{Implies}, double] (11,-1.5) -- (11,-1.8);
}
\node(f) at (11,-2.5){Terminate};
\end{tikzpicture}
\end{center}
The resulting multisegment is $\{[0,1], [-1,0]\}$. This orbit corresponds to the representation $Q([0,1],[-1,0])$. In this way, we can compute the multisegments for all orbits in the Vogan variety $V$ from in Example \ref{h and v}. 

\begin{center}
\begin{tabular}{ |c|c| } 
 \hline
$\Pi_{\lambda}(\GL_4(F))$
 & Rank triangle \\
 \hline

 $Q([-1,1],[0])$ &   
\begin{tikzpicture}
\node(A) at (0,0.5) {$-1$};
\node(B) at (0.5,0.5) {$0$};
\node(C) at (1,0.5) {$1$};
\draw (-0.12,0.33) -- (1.12,0.33);
\node(A) at (0,0.15) {$1$};
\node(B) at (0.5,0.15) {$2$};
\node(C) at (1,0.15) {$1$};
\draw (-0.12,0) -- (1.12,0);
\node(D) at (0.25,-0.2) {$1$};
\node(E) at (0.75,-0.2) {$1$};
\node(d) at (0.5,-0.5) {$1$}; \end{tikzpicture}\\

\hline
$Q([0,1],[-1,0])$ &   
\begin{tikzpicture}
\node(A) at (0,0.5) {$-1$};
\node(B) at (0.5,0.5) {$0$};
\node(C) at (1,0.5) {$1$};
\draw (-0.12,0.33) -- (1.12,0.33);
\node(A) at (0,0.15) {$1$};
\node(B) at (0.5,0.15) {$2$};
\node(C) at (1,0.15) {$1$};
\draw (-0.12,0) -- (1.12,0);
\node(D) at (0.25,-0.2) {$1$};
\node(E) at (0.75,-0.2) {$1$};
\node(d) at (0.5,-0.5) {$0$}; \end{tikzpicture}\\
\hline
$Q([1], [0], [-1,0])$ &   
\begin{tikzpicture}
\node(A) at (0,0.5) {$-1$};
\node(B) at (0.5,0.5) {$0$};
\node(C) at (1,0.5) {$1$};
\draw (-0.12,0.33) -- (1.12,0.33);
\node(A) at (0,0.15) {$1$};
\node(B) at (0.5,0.15) {$2$};
\node(C) at (1,0.15) {$1$};
\draw (-0.12,0) -- (1.12,0);
\node(D) at (0.25,-0.2) {$1$};
\node(E) at (0.75,-0.2) {$0$};
\node(d) at (0.5,-0.5) {$0$}; \end{tikzpicture}\\
\hline

$Q([0,1],[0],[-1])$ &   
\begin{tikzpicture}
\node(A) at (0,0.5) {$-1$};
\node(B) at (0.5,0.5) {$0$};
\node(C) at (1,0.5) {$1$};
\draw (-0.12,0.33) -- (1.12,0.33);
\node(A) at (0,0.15) {$1$};
\node(B) at (0.5,0.15) {$2$};
\node(C) at (1,0.15) {$1$};
\draw (-0.12,0) -- (1.12,0);
\node(D) at (0.25,-0.2) {$0$};
\node(E) at (0.75,-0.2) {$1$};
\node(d) at (0.5,-0.5) {$0$}; \end{tikzpicture}\\
\hline
$Q([1],[0],[0],[-1])$ &   
\begin{tikzpicture}
\node(A) at (0,0.5) {$-1$};
\node(B) at (0.5,0.5) {$0$};
\node(C) at (1,0.5) {$1$};
\draw (-0.12,0.33) -- (1.12,0.33);
\node(A) at (0,0.15) {$1$};
\node(B) at (0.5,0.15) {$2$};
\node(C) at (1,0.15) {$1$};
\draw (-0.12,0) -- (1.12,0);
\node(D) at (0.25,-0.2) {$0$};
\node(E) at (0.75,-0.2) {$0$};
\node(d) at (0.5,-0.5) {$0$}; \end{tikzpicture}\\
\hline
\end{tabular}
\end{center}
\vspace{2mm}
The open orbit in $V$ consists of elements of full rank. From the table above, we see that the open orbit corresponds to the representation $Q([-1,1],[0])$. On the other hand, having directly computed the multisegment \{[-1,0],[0,1]\} for $\phi_{\psi}$ from Example \ref{h and v}, we now have its rank triangle.

\end{example}

\section{Zelevinsky involution on multisegments and orbits}
\label{zel-inv}
We now work with an involution on orbits, often referred to as the Zelevinsky involution, which interchanges the Steinberg and trivial representations. This appears in \cite{Z2}, \cite{zelevinskii1981p}, and is expanded upon in \cite{MW:involution}. We first write down an algorithm to compute the involution $\alpha \mapsto \Tilde{\alpha}$ following \cite{MW:involution} and then prove an important combinatorial lemma (Lemma \ref{mainlemma}). Finally, we see the analogous notions on orbits.

Let $\alpha$ denote a multisegment. This means that $\alpha$ is a collection of segments of the form $[b,e]$, where $b$ and $e$ are integers or half integers and are referred to as the base value and end values of the segment respectively. 

\subsection{Moeglin-Waldspurger algorithm}
\label{mw_alg}
In this section, write down the Moeglin-Waldpurger algorithm \cite{MW:involution} to compute the Zelevinsky involution on multisegments for $\GL_n$. For this we recall what it means for one segment to precede the other. Let $\Delta_1 = [b_1,e_1]$ and $\Delta_2=[b_2,e_2]$ be two segments.  


\begin{definition}
$\Delta_1$ is said to \textit{precede} $\Delta_2$ if and only if \[b_1 < b_2 \text{, } e_1<e_2, \text{ and } b_2 \leq e_1 + 1. \]
\end{definition}

We list the steps to compute $\Tilde{\alpha}$ below. 

First, we compute the segment $M(\alpha)$ associated to $\alpha$:
\begin{enumerate}
    \item Set $e$ to be the largest end value that appears in any segment. Set $m:=e$.  
    \item Consider all segments in $\alpha$ with end value $m$. Among these, choose a segment with the largest base value and call it $\Delta_m$. 
    \item Consider the set of all segments in $\alpha$ that precede $\Delta_m$ with end value $m-1$. If this is empty, go to step $(5)$. Otherwise, choose a segment from this set with the largest base value and call it $\Delta_{m-1}$. 
    \item Set $m:=m-1$ and go to step (3). 
    \item Return $M(\alpha):=[m,e]$.
\end{enumerate}

Next, we inherit the following notation from the above procedure: for any $i \in \{m,m+1,\cdots,e\}$, set $\Delta_i$ as the segment whose base value is the highest among segments with end value $i$. Define $\alpha \setminus M(\alpha)$ to be the multisegment obtained by removing $i$ from $\Delta_i$ for all $m\leq i \leq e$. Then, the dual $\Tilde{\alpha}$ is recursively defined via \[\Tilde{\alpha}:= (M(\alpha), \widetilde{\alpha\setminus M(\alpha)}).\] 
The fact that this procedure gives you an involution comes from \cite{MW:involution}. In particular, we have the useful property that $\Tilde{\Tilde{\alpha}}=\alpha$. 

\begin{example} \label{multisegment inv exam}
Let us continue working in the setting of Example \ref{h and v} and compute the involution on multisegments that appear:
\begin{center}
    
\begin{tabular}{|c|c|}
    \hline
    $\alpha$ &  $\Tilde{\alpha}$\\
    \hline 
    $\{[-1,1], [0]\}$ & $\{[-1],[0],[0],[1]\}$ \\
    \hline
    $\{[-1,0], [0,1]\}$ & $\{[-1,0], [0,1]\}$ \\
    \hline
    $\{[1],[0],[-1,0]\}$ & $\{[0,1],[0],[-1]\}$\\
    \hline
\end{tabular}
\end{center}

\end{example}

\subsection{Lemma about simple multisegments} The aim of this section is to prove Lemma \ref{mainlemma}. We establish some terminology and notation associated to a multisegment $\alpha$. The length of a segment $[b,e]$ is simply the value $e-b+1$. 

\noindent $L_{\alpha}$:= Length of the longest segment in $\alpha$. \\
$n_{\alpha}$:= number of segments in $\alpha$ \\
$c_{\alpha}$:= minimum number of segments that $\bigcup_{\Delta \in \alpha}\Delta$ can be broken into. 

\begin{lemma}
\label{lemma1}
Let $\alpha$ and $\beta$ be any multisegments. If $\alpha \leq \beta$, then
\begin{enumerate}
    \item $L_{\alpha} \leq L_{\beta}$.
    \item $n_{\alpha} \geq n_{\beta} .$
    \item $n_{\Tilde{\alpha}} \geq L_{\alpha}$ and $n_{\alpha} \geq L_{\Tilde{\alpha}}$. 
\end{enumerate}
\end{lemma}
\begin{proof} 
\begin{enumerate}
   \item Since $\beta$ is obtained by performing elementary operations on segments in $\alpha$, the longest segment in $\beta$ is either a segment from $\alpha$ or a union of two segments in $\alpha$. Therefore, $L_{\beta} \geq L_{\alpha}.$
    \item Every pair of segments in $\alpha$ is replaced by at most two segments to form $\beta$. Therefore $n_{\beta} \leq n_{\alpha}$.
    \item Fix a segment in $\alpha$. From the Moeglin-Waldspurger algorithm, one can observe that each element of that segment appears in a distinct segment in $\Tilde{\alpha}$. In particular, this is true for the longest segment of $\alpha$. Thus, we get $n_{\Tilde{\alpha}} \geq L_{\alpha}$. Since $\Tilde{\Tilde{\alpha}}=\alpha$, we get the other inequality. 
\end{enumerate}
\end{proof}
\begin{definition}
$\alpha$ is a simple multisegment if it is of the form 
\[\{[b,e],[b+1,e+1],\cdots,[b+n-1,e+n-1]\}.\]
\end{definition}
\noindent Observe that $n$ in the definition above is equal to $n_{\alpha}$. 
\begin{lemma}
\label{lemma2}
If $\alpha$ is a simple multisegment, then
\begin{enumerate}
    \item $\Tilde{\alpha}$ is simple.
    \item $n_{\Tilde{\alpha}}=L_{\alpha}$.
    \item $c_{\alpha}=1$.
\end{enumerate}

\end{lemma}
\begin{proof}
\begin{enumerate}
    \item Applying the Moeglin-Waldspurger algorithm, one can explicitly compute 
    \[\Tilde{\alpha}=\{[e-k,e+n_{\alpha}-k-1] : 0\leq k \leq e-b\}.\] Clearly, it is simple. 
     
    \item Observe that the length of every segment in $\alpha$ is the same, and $L_{\alpha}=e-b+1$. Observing how the index $k$ varies in part $(1)$ tells us that $n_{\alpha}=e-b+1$.
    \item $c_{\alpha}=1$ because
    \[\bigcup_{\Delta \in \alpha}\Delta = [b,e+n-1].\]
    
\end{enumerate}
\end{proof}
\begin{lemma}
\label{lemma3}
If $L_{\Tilde{\alpha}}=n_{\alpha}$, $\alpha \leq \beta$, and $\Tilde{\alpha} \leq \Tilde{\beta}$, then 
\[n_{\alpha}=n_{\beta}=L_{\Tilde{\alpha}}=L_{\Tilde{\beta}}.\]
The above statement also holds true if $\alpha$ is interchanged with $\Tilde{\alpha}$ and $\beta$ with $\Tilde{\beta}$. 
\end{lemma}
\begin{proof} By hypothesis, $\Tilde{\alpha} \leq \Tilde{\beta}$. Therefore,
    
\begin{align*}
    L_{\Tilde{\alpha}} &\leq L_{\Tilde{\beta}}, \text{ (from Lemma \ref{lemma1})}\\
    &\leq n_{\Tilde{\Tilde{\beta}}}, \text{ (from Lemma \ref{lemma1}) }\\
    &= n_{\beta}, \text{ (since $\Tilde{\Tilde{\beta}}=\beta$)} \\
    &\leq n_{\alpha}, \text{ (hypothesis and Lemma \ref{lemma1})}\\
    &=L_{\Tilde{\alpha}}. \text{ (hypothesis)}
\end{align*}

\end{proof}

\begin{lemma}
\label{lemma4}
Let $\alpha$ be a simple multisegment and $\beta$ any multisegment such that $\alpha \leq \beta$ and $L_{\alpha}\leq L_{\beta}$. Then, $\alpha=\beta$.
\end{lemma}
\begin{proof}
Suppose $\alpha < \beta$. Then, there is a segment in $\beta$ of the form $\Delta_1 \cup \Delta_2$ strictly containing $\Delta_1$ and $\Delta_2$, where $\Delta_1, \Delta_2 \in \alpha$. Now $\alpha$ is simple, so every segment has the same length and that length is equal to $L_{\alpha}$. Thus, the strict containment above implies
\[L_{\beta} \geq \text{ length of } \Delta_1 \cup \Delta_2 > \text{ length of } \Delta_1 = L_{\alpha}, \]
which contradicts the hypothesis.
\end{proof}
Now we are ready to state the primary result of this section. 
\begin{lemma}
\label{mainlemma}
Suppose $\alpha$ is a simple multisegment and $\beta$ is any multisegment satisfying $\alpha \leq \beta$ and $\Tilde{\alpha} \leq \Tilde{\beta}$, then $\alpha=\beta$. 
\end{lemma}
\begin{proof}
$\alpha$ is simple, so from Lemma \ref{lemma2} $n_{\Tilde{\alpha}}=L_{\alpha}$. Lemma \ref{lemma3} implies $L_{\alpha}=L_{\beta}$. The hypothesis of Lemma \ref{lemma4} is satisfied and thus $\alpha=\beta$. 
\end{proof}

\subsection{Involution on orbits}
\label{dualorbits}
Following \cite{CFMMX}*{Section 6}, we explain the notion of a dual orbit which is used to define an analogous involution on orbits $C \mapsto \widehat{C}$ of $V_{\lambda}$. 

From \eqref{vlambdadef}, we may write 
\[V=\{x \in \mathfrak{gl}_n(\CC) : \operatorname{Ad}(\lambda(\Frob)x=qx\}. \]
Using the killing form on $\mathfrak{gl}_n(\CC)$, the dual variety $V^*$ may be realized as
\[V^*:=\{x \in \mathfrak{gl}_n(\CC) : \operatorname{Ad}(\lambda(\Frob)x=q^{-1}x\}.\]

We have a presentation of $V^*$ analogous to \eqref{vlambda}:
\[V^* =  \Hom(E_1,E_0)\times \Hom(E_{k-1},E_{k-2})\times \cdots \times \Hom(E_{k},E_{k-1}), \]
with an $H$-action
\[(h_k,\ldots,h_1,h_0) \ldots (y_{10}, y_{21},\ldots, y_{kk-1} )=(h_0y_{10}h_1^{-1}, h_1y_{21}h_2^{-1}, \ldots , h_{k-1} y_{kk-1} h_k^{-1}),\]
where $(h_k,\ldots,h_1,h_0) \in H$ and $(y_{10}, y_{21},\ldots, y_{kk-1} ) \in V^*$. 

The cotangent variety $T^*(V)$ for $V$ is simply $V \times V^*$. For $x \in V$ and $y \in V^*$, $[x,y]=xy-yx$ is the usual Lie bracket from $\mathfrak{gl}_n(\CC)$. The conormal variety, $\Lambda \subseteq T^*(V)$, is defined as the kernel of the map $[\cdot,\cdot]:T^*(V) \xrightarrow{} \operatorname{Lie}(H)$. 

By \cite{CFMMX}*{Proposition 6.3.1}, for each $H$-orbit $C$ in $V$, 
\begin{equation}
 \label{conormal}   
\Lambda_C:=\{(x,y) \in \Lambda : x \in C\}
\end{equation}
is the conormal bundle to $C$. By identifying $V^{**} \simeq V$, for an $H$-orbit $C^*$ in $V^*$ we may write
\[\Lambda_{C^*}=\{(x,y) \in \Lambda : x \in C^*\}\]
where $\Lambda_{C^*}$ is the conormal bundle to $C^*$. 

For any $C$, there exists a unique orbit $C^*$ in $V^*$ with $\bar{\Lambda}_C \simeq \bar{\Lambda}_{C^*}$; see \cite{Pyasetskii}*{Corollary 2}. The map $C \mapsto C^*$ defines a bijection the set of $H$-orbits in $V$ and the set of $H$-orbits in $V^*$. Using this, we may define
\[\widehat{C}:=\prescript{t}{}{C^*},\]
where $y \mapsto \prescript{t}{}{y}$ is matrix transposition. Then $C \mapsto \widehat{C}$ defines an involution on the orbits. 

In section \ref{mw_alg}, we wrote down the Zelevinsky involution on multisegments and in this section we see an involution on orbits. Since multisegments and orbits correspond to each other via rank triangles as seen in Section \ref{ranktri-multiseg}, it is natural to ask whether this involution is compatible under this correspondence. Crucially, it is. More precisely, if the multisegment $\alpha$ corresponds to the orbit $C$, then $\Tilde{\alpha}$ corresponds to $\widehat{C}$. See \cite{knight1996representations} for more details. 

\begin{example}
From Example \ref{multisegment inv exam}, we see that the involution applied to $\{[-1,1],[0]\}$ gives $\{[1],[0],[0],[-1]\}$. From Example \ref{multisegment-rk tri example}, we see that the involution interchanges the open orbit (with elements of full rank) with the closed point (origin). 
\end{example}
\section{Vogan's conjecture for simple Arthur parameters}
\label{proof}

Recall the simple Arthur parameter $\psi$ defined in \eqref{arthurpar}, the Langlands parameter $\phi_{\psi}$ coming from it defined in \eqref{langlandspar}, and its infinitesimal parameter $\lambda$ defined in \eqref{infpar}. We remind the reader that $G=\GL_n$. Now, $\pi \in \Pi(G(F))$ is said to have the infinitesimal parameter $\lambda$ if its corresponding Langlands parameter $\phi$ has the infinitesimal parameter $\lambda$. Let $\Pi_{\lambda}(G(F))$ denote all such $\pi$ and $\Phi_{\lambda}(G(F))$ denote equivalence classes of all such $\phi$. As outlined in sections \ref{vogan langlands corresp} and \ref{modspace}, equivalence classes of smooth irreducible representations of $G(F)$ with infinitesimal parameter $\lambda$ correspond to equivalence classes of simple objects in $\Perv_H(V)$. This bijection is summarized below:
\[    
\begin{tikzcd}
	{\Pi_{\lambda}(G(F))} & {\Phi_{\lambda}(G(F))} & {H \operatorname{-orbits} \operatorname{ in}V} & {\operatorname{Perv}_H(V)^{\operatorname{simple}}_{/\operatorname{iso}},} & {} \\[-10pt]
	{\pi \hspace{2mm}} & \phi & {\hspace{2mm}C_{\phi} \hspace{2mm}} & {\hspace{2mm}\mathcal{P}(\pi)=\mathcal{IC}(\mathbb{1}_{C_{\phi}}).}
	\arrow[from=1-1, to=1-2]
	\arrow[from=1-2, to=1-3]
	\arrow[maps to, from=2-1, to=2-2]
	\arrow[maps to, from=2-2, to=2-3]
	\arrow[from=1-3, to=1-4]
	\arrow[maps to, from=2-3, to=2-4]
\end{tikzcd}
\]
Thus, there is a unique orbit $C_{\phi_{\psi}}$ in $V$ attached to $\phi_{\psi}$. We shorten this notation to $C_{\psi}$. 
This orbit in turn corresponds to a simple multisegment - one can see this directly by writing down the multisegment for the Langlands parameter $\phi_{\psi}$ when $\psi(w,x,y)=\op{Sym}^d(x) \otimes \op{Sym}^a(y)$: 
\begin{equation}
 \label{simple mult}   
\left\{\left[\tfrac{-(d+a)}{2},\tfrac{-(d+a)}{2}+d\right], \left[\tfrac{-(d+a)}{2}+1,\tfrac{-(d+a)}{2}+d+1\right],\ldots, \left[\tfrac{-(d+a)}{2}+a,\tfrac{-(d+a)}{2}+d+a\right]\right\}.
\end{equation}
The dictionary between Langlands parameters and multisegments is written down in many places, see \cite{kudla1994local}*{Section 4.2}, for example. In Section \ref{ranktri-multiseg} we saw that rank triangles and multisegments correspond to each other and this correspondence respects the partial order on both sides. In section \ref{dualorbits}, we see that this correspondence respects the involution on both sides. Thus, we may adapt Lemma \ref{mainlemma} as follows:
\begin{lemma}
\label{orbitduallemma}
Let $C_{\psi}$ be the orbit corresponding to a simple Arthur parameter as explained above. Let $C$ be any $H$-orbit in $V$. If $C_{\psi} \leq C$ and $\widehat{C}_{\psi} \leq \widehat{C}$, then $C_{\psi}=C$.  
\end{lemma}

\begin{lemma}
\label{hatpsilemma}
Let $C_{\psi}$ be the orbit corresponding to a simple Arthur parameter. Then
\[
\widehat{C}_{\psi} 
= 
C_{\widehat{\psi}} ,
\]
where ${\widehat{\psi}}$ is the the simple Arthur parameter obtained by interchanging the role of the two $\op{SL}_2(\CC)$'s in $W_F''$: \[\widehat{\psi}(w,x,y):=\psi(w,y,x).\]
\end{lemma}

\begin{proof}
Since $\psi$ is simple, $C_{\psi}$ corresponds to the multisegment written in \eqref{simple mult}. The dual of this multisegment, in the sense of \ref{mw_alg}, is
\begin{equation*}
    \left\{\left[\tfrac{-(d+a)}{2},\tfrac{-(d+a)}{2}+a\right], \left[\tfrac{-(d+a)}{2}+1,\tfrac{-(d+a)}{2}+a+1\right],\ldots, \left[\tfrac{-(d+a)}{2}+d,\tfrac{-(d+a)}{2}+d+a\right]\right\}.
\end{equation*}
By inspection, this multisegment comes from the simple Arthur parameter $(w,x,y) \mapsto \op{Sym}^a(x)\otimes \op{Sym}^d(y)$. From the discussion in \ref{dualorbits}, this multisegment also corresponds to $\widehat{C}_{\psi}$, as required.
\end{proof}

We recall the functor defined in \eqref{evs}: 
\[    \Evs_{\psi}: \Perv_{H_{\lambda_\psi}}(V_{\lambda_\psi}) \xrightarrow{} \Rep(A_\psi),
.\]
For the Langlands parameter $\phi_{\psi}$, we recall the ABV-packet from \eqref{abvpacket}:
\begin{equation*}
    \Pi_{\phi_{\psi}}^{\ABV}(G(F)):=\{\pi \in \Pi_{\lambda}(G(F)): \operatorname{Evs}_{\psi}(\mathcal{P}(\pi))\neq 0 \}.
\end{equation*} 

We now state and prove the main theorem of this paper.
\begin{theorem}\label{ wow main theorem}
Let $\psi$ be a simple Arthur parameter of $G=\GL_n$, $\phi_{\psi}$ its corresponding Langlands parameter, $\Pi_{\psi}(G(F))$ the Arthur packet attached to $\psi$, and $\Pi_{\phi_{\psi}}^{\ABV}(G(F))$ the ABV-packet attached to $\phi_{\psi}$. Then,

\[ \Pi_{\phi_{\psi}}^{\ABV}(G(F))=\Pi_{\psi}(G(F)).\]
\end{theorem}

\begin{proof}
The local Langlands correspondence tells us that $L$-packets are singletons. As discussed in \ref{apackets}, the Arthur packet $\Pi_{\psi}(G(F))$ is a singleton and it coincides with the $L$-packet for $\phi_{\psi}$. We may write
\[\Pi_{\psi}(G(F))=\Pi_{\phi_{\psi}}(G(F))=\{\pi_{\psi}\}.\] 

Let $C_{\psi}$ be the unique orbit associated to the Langlands parameter $\phi_{\psi}$. By \cite{CFMMX}*{Theorem 7.19}, $\Evs_{\psi} \mathcal{IC}(\mathbb{1}_{C_{\psi}}) \neq 0$. By definition, this implies 
\[\Pi_{\phi_{\psi}}(G(F)) \subseteq \Pi_{\phi_{\psi}}^{\ABV}(G(F)),\]
and therefore \[\Pi_{\psi}(G(F)) \subseteq \Pi_{\phi_{\psi}}^{\ABV}(G(F)).\]
To prove the theorem, it is sufficient to prove that  $\Pi_{\phi}^{\ABV}(G(F))$ is a singleton, \textit{i.e.}, it contains no representation other that $\{\pi_{\psi}\}$. 

Suppose we have a representation $\pi \in \Pi_{\phi}^{\ABV}(G(F))$. By definition, this means
\begin{equation}
\label{defeq}
\operatorname{Evs}_{\psi}(\mathcal{P}(\pi)) \neq 0.
\end{equation}
From the discussion in sections \ref{vogan langlands corresp} and \ref{modspace}, we know that $\pi$ corresponds to an $H$-orbit $C$ in $V$ and $\mathcal{P}(\pi)=\mathcal{IC}(\mathbb{1}_C)$. We may rephrase \eqref{defeq} as
\begin{equation}
\label{evsnonzero}
    \operatorname{Evs}_{\psi}\mathcal{IC}(\1_C) \neq 0.
\end{equation}
Using \cite{KS_sing}*{Proposition 3.2.1} and Lemma~\ref{hatpsilemma}, \eqref{defeq} also implies that 
\begin{equation}
\label{evsdualnonzero}
    \operatorname{Evs}_{\widehat{\psi}}\mathcal{IC}(\1_{\widehat C}) \neq 0.
\end{equation}
Again from \cite{KS_sing}*{Proposition 3.2.1}, \eqref{evsnonzero} implies that $C_{\psi} \leq C$ and \eqref{evsdualnonzero} implies $\widehat{C}_{\psi} \leq \widehat{C}$. The hypothesis of Lemma \ref{orbitduallemma} is satisfied, so we may conclude that $C_{\psi}=C$. The corresponding representations are the same, as required.

\end{proof}

\section{Vogan's conjecture for irreducible Arthur parameters}\label{irreducible a parameter section}
We now extend the result in Theorem \ref{ wow main theorem} to an \textit{irreducible} Arthur parameter of $\GL_n$, {\it i.e.},
of the form
\begin{equation}\label{irred a par}
\psi(w,x,y)= \rho(w) \otimes \op{Sym}^d(x) \op{Sym}^a(y),\end{equation}
where $\rho$ is an irreducible admissible representation of $W_F$ of dimension $m$, and $a$ and $d$ are non-negative integers. First note that $\rho=\rho_0 \otimes |\cdot |^s$, where $\rho_0$ is of Galois type and $|\cdot |$ is the absolute value on the Weil group as explained in section \ref{modspace}. 

We apply the procedure of hyper-unramification as explained in \cite{CFMMX}*{Section 5} on the infinitesimal parameter $\lambda$. Note that $\lambda(w)=\rho_0(w) \otimes |w|^s \otimes \op{Sym}^d(d_w) \otimes \op{Sym}^a(d_w)$ where $d_w$ is defined in Section \ref{modspace}. For any $X \in \op{SL}_2(\CC)$, set $\eta(X):=\op{Sym}^d(X) \otimes \op{Sym}^a(X)$. 

\begin{lemma}\label{hyperbolic-elliptic}
The semisimple element $\lambda(\Frob)$ has the decomposition $\lambda(\Frob)=s_{\lambda}t_{\lambda}$ where \[s_{\lambda}= \rho_0\begin{pmatrix}1&0\\0&1\end{pmatrix} \otimes |\Frob|^{\op{Re}(s)} \otimes \eta(d_{\Frob})\]
is the hyperbolic part and 
\[t_{\lambda}= \rho_0(\Frob) \otimes |\Frob|^{i\op{Im}(s)} \otimes \eta \begin{pmatrix}1&0\\0&1\end{pmatrix}\] is the elliptic part.
\end{lemma}
\begin{proof}
Clearly, $s_{\lambda}t_{\lambda} = t_{\lambda}s_{\lambda}=\lambda(\Frob)$. To show that $s_{\lambda}$ is hyperbolic, it suffices to show that its eigenvalues are positive real numbers. Indeed, after $s_{\lambda}$ is identified with a matrix, its eigenvalues are of the form $q^{\op{Re}s+\frac{ a+d}{2}+j}$ where $j$ is an integer. This can be seen from direct calculation aided by Equation \ref{lambda eigenvalues}. To show that $t_{\lambda}$ is elliptic, it suffices to show that its eigenvalues have complex norm 1. Indeed, $t_{\lambda}$ can be identified with the (block) matrix \[\begin{bmatrix} \rho_0(\Frob)q^{i\op{Im}s}&0& \cdots &0 \\
0&\rho_0(\Frob)q^{i\op{Im}s}& \cdots &0\\
\vdots&  & \ddots& \\
0&\cdots&0&\rho_0(\Frob)q^{i\op{Im}s}
\end{bmatrix},\]
so complex norm of its eigenvalues are determined by those of $\rho_0(\Frob)$. However, $\rho_0$ is of Galois-type which means it has finite image. Thus $\rho_0(\Frob)$ is a matrix of finite order, whose eigenvalues must be roots of unity. 

\end{proof}
The process of hyper-unramification described in \cite{CFMMX}*{Section 5.3} adapted to our case gives us an unramified parameter 
$\lambda_{nr}:W_F \xrightarrow[]{} J_{\lambda}$, 
where $J_{\lambda}$ is $Z_{\widehat{G}}(\lambda|_{I_F}, t_{\lambda})$ whose image of frobenius is the hyperbolic element $s_{\lambda}$. Using \cite{CFMMX}*{Lemma 5.3} we have that \[V_{\lambda_{nr}}=V_{\lambda}, \text{ and } H_{\lambda_{nr}}=H_{\lambda}^0. \]

\begin{lemma}\label{j/h lambda proof}
Suppose $\psi$ is an irreducible Arthur parameter as defined in Equation \ref{irred a par}. Then,
\begin{enumerate}
    \item $J_{\lambda}=\GL_{n/m}(\CC)$. 
    \item $H_{\lambda_{nr}}=H_{\lambda}$
\end{enumerate}
\end{lemma}
\begin{proof}
\begin{enumerate}
    \item By definition, $J_{\lambda}=Z_{\hat{G}}(\lambda|_{I_F}, t_{\lambda})$. Now $\lambda(w)|_{I_F}=\rho(w)|_{I_{F}} \otimes \eta(d_w)|_{I_F}$. Since $I_F$ consists of elements of absolute value 1, $\lambda(w)|_{I_F}=\rho_0(w)|_{I_F}\otimes \eta \begin{pmatrix} 1&0\\0&1 \end{pmatrix}$. As $\rho_0$ is of Galois-type, it is determined by a finite list of elements. In particular, there are elements $\sigma_1,\sigma_2,\ldots, \sigma_k$ in $I_F$ so that their images (block matrices) of the form 
    \[\begin{bmatrix}\rho_0(\sigma_i)&0&\cdots&0\\
    0&\rho_0(\sigma_i)& \cdots&0\\
    \vdots&&\ddots&\\
    0&0&\cdots& \rho_0(\sigma_i)
    \end{bmatrix}\]
    determine $\rho_0|_{I_F}$. 
    \begin{align*}
        J_{\lambda}&=Z_{\GL_n(\CC)}(\lambda(w)|_{I_F}, t_{\lambda})\\
                   &=\{X \in \GL_n(\CC) : t_{\lambda }Xt_{\lambda}^{-1}=X \text{ and } \lambda(\sigma_i)X\lambda(\sigma_i)^{-1}=X  \text{ for } 1 \leq i \leq k \} \\
                   &=\{X=(X_{rs})_{r,s} \in \GL_n(\CC) \text{ where } X_{rs} \in \op{Mat}_m(\CC): \rho_0(\Frob) X_{rs} \rho_0(\Frob)^{-1}=X_{rs} \text{ and } \rho_0(\sigma_i) X_{rs} \rho_0(\sigma_i)^{-1}=X_{rs}  \}\\
                   &=\{X=(X_{rs})_{r,s} \in \GL_n(\CC) \text{ where } X_{rs} \in \op{Mat}_m(\CC): X_{rs} \in Z_{\op{Mat}_n(\CC)}(\rho_0)\}\\
                   &=\{X=(c_{rs}\op{Id}_m)_{r,s} \in \GL_n(\CC) \text{ where } c_{rs} \in \CC\}\\
                   &\simeq \GL_{n/m}(\CC).
    \end{align*}
    The first equality is by definition, the second by the fact that $\lambda|_{I_F}$ is determined by $\sigma_is$ as explained above, the third follows by writing down $X, t_{\lambda}$, and $\lambda(\sigma_i)s$ as block matrices, the fourth follows from the fact that $\rho_0$ is completely determined by $\Frob$ and $\sigma_i$s, and the fifth follows from Schur's lemma applied to the irreducible representation $\rho_0$ for every $r,s$. The isomorphism in the sixth step comes from realizing $(c_{rs}\op{Id}_m)_{r,s}$ as $(c_{rs})_{r,s}\otimes \op{Id}_m$ whose determinant is $\op{det}((c_{rs})_{r,s})^m$ using the determinant of the Kronecker product of two matrices. 
    \item As centralizers in $\GL_n(\CC)$ are connected, $H_{\lambda}$ is connected so $H_{\lambda_{nr}}=H_{\lambda}^0=H_{\lambda}$.
    
\end{enumerate}
\end{proof}

The above lemma shows that if we replace $J_{\lambda}$ with the isomorphic  general linear group, we may interpret the unramified infinitesimal parameter as a homomorphism $W_F\xrightarrow{} \GL_{n/m}(\CC)$ whose image at the frobenius can now be thought of as $|\Frob|^{\op{Re}s}\otimes \eta(d_{\Frob})$, and whose Vogan variety and group acting on it remain the same as that of $\lambda$. Also observe that removing the factor $|\Frob|^{\op{Re}s}$ from this infinitesimal parameter does not change the variety and the group acting on it for trivial reasons. We collect these observations on hyper-unramification of $\psi$ as defined in \ref{irred a par} in the following lemma.

\begin{lemma}\label{unramification lemma}
Suppose $\psi$ is an \textit{irreducible} Arthur parameter of $\GL_n$, {\it i.e.}, \[\psi(w,x,y)=\rho(w) \otimes \op{Sym}^d(x) \otimes \op{Sym}^a(y),\]
where $\rho$ is an irreducible representation of $W_F$ in $\GL_m(\CC)$. We obtain the unramified parameter \[\lambda_{nr}:W_F\xrightarrow{} \GL_{n/m}(\CC)\]
defined via $w \mapsto \op{Sym}^d(d_w)\otimes \op{Sym}^a(d_w)$, with 
\[V_{\lambda_{nr}}\simeq V_{\lambda}, \text{ and } H_{\lambda_{nr}}\simeq H_{\lambda},\]
where these isomorphisms are obtained by an application of Schur's lemma, as in the proof of Lemma \ref{j/h lambda proof} (1).
\end{lemma}

The upshot is that $\lambda_{nr}$ can be thought of as the infinitesimal parameter of $\psi_{nr}(w,x,y):= \op{Sym}^d(x)\otimes \op{Sym}^a(y)$, after an application of the isomorphisms $V_{\lambda_{nr}}\simeq V_{\lambda}$ and $H_{\lambda_{nr}}\simeq H_{\lambda}$ that are induced by taking scalar matrices to their respective scalars in a smaller matrix group.
Since the variety and group action are unchanged, the category $\op{Per}_{H_{\lambda_{nr}}}(V_{\lambda_{nr}})=\op{Per}_{H_{\lambda}}(V_{\lambda})$. As a result, we directly see that the isomorphisms give us a bijection $C \mapsto C_{nr}$ of orbits which preserves the ordering and the Zelevinsky involution. The geometric side of the Vogan-Langlands correspondence can thus be reduced to the study of the unramified infinitesimal parameter. This paves the way for generalizing Theorem \ref{ wow main theorem} from simple Arthur parameters to irreducible Arthur parameters. 

\begin{theorem}\label{bow wow main theorem}
Suppose $\psi$ is an \textit{irreducible} Arthur parameter of $G=\GL_n$, {\it i.e.}, \[\psi(w,x,y)=\rho(w) \otimes \op{Sym}^d(x) \otimes \op{Sym}^a(y),\]
where $\rho$ is an irreducible representation of $W_F$ on $\GL_m(\CC)$. Let $\Pi_{\psi}(G(F))$ denote the local Arthur packet of $\psi$ and $\Pi_{\phi_{\psi}}^{\ABV}(G(F))$ denote the ABV-packet attached to $\phi_{\psi}$. Then, 
\[\Pi_{\psi}(G(F))=\Pi_{\phi_{\psi}}^{\ABV}(G(F)).\]
\end{theorem}
\begin{proof}
From Lemma \ref{unramification lemma}, $V_{\lambda}\simeq V_{\lambda_{nr}}$ and $H_{\lambda}\simeq H_{\lambda_{nr}}$. We have a bijection $C \mapsto C_{nr}$ between orbits in $V_{\lambda}$ and $V_{\lambda_{nr}}$ which preserves partial order and the Zelevinsky involution on orbits. This gives a bijection $\pi \mapsto \pi_{nr}$ between $\Pi_{\lambda}(G(F))$ and $\Pi_{\lambda_{nr}}(G(F))$ preserving the partial order and involution on multisegments such that $(C_{\pi})_{nr}=C_{\pi_{nr}}$. In particular, $(C_{\psi})_{nr}=C_{\psi_{nr}}$. Since $\op{Evs}_{\psi}$ is purely dependent on the geometry, we have that $\op{Evs}_{\psi}\mathcal{IC}(\mathbb{1}_{C_{\pi}}) \neq 0$ for $\pi \in \Pi_{\lambda}(G(F))$ if and only if $\op{Evs}_{\psi_{nr}}\mathcal{IC}(\mathbb{1}_{C_{\pi_{nr}}}) \neq 0$ for the corresponding $\pi_{nr} \in \Pi_{\lambda_{nr}}(G(F))$. Since the ABV-packet for $\psi_{nr}$ is a singleton from Theorem \ref{ wow main theorem}, same is true for the ABV packet for $\psi$.
\end{proof}
\begin{Remark}
The process of hyper-unramification taking $\psi \mapsto \psi_{nr}$ gives us a bijection between $\Pi_{\lambda}(G(F))$ and $\Pi_{\lambda_{nr}}(G(F))$ which can be written down explicitly in terms of multisegments. Let $\pi_{\rho}$ be the supercuspidal representation of $\GL_m(F)$ corresponding to the irreducible representation $\rho$. The correspondence at the level of multisegments is induced by taking any segment of the form $[\pi_{\rho}(b), \pi_{\rho}(e)]$ to the segment $[b,e]$. 
\end{Remark}


\end{document}